\documentclass{article}
\usepackage{amsmath}
\usepackage{amssymb}
\usepackage{amsthm}
\usepackage{booktabs}
\usepackage[shortlabels]{enumitem}

\newtheorem{thm}{Theorem}
\newtheorem{lem}[thm]{Lemma}
\newtheorem{cor}[thm]{Corollary}
\newtheorem{qn}[thm]{Question}

\title{On Tribonacci Sequences}
\date{Saturday 28, January 2023}
\author{Luke Pebody}
\begin{document}
\maketitle
\begin{abstract}
Let a tribonacci sequence be a sequence of integers satisfying $a_k=a_{k-1}+a_{k-2}+a_{k-3}$
for all $k\ge 4$. For any positive integers $k$ and $n$, denote by $f_k(n)$ the number of tribonacci sequences with $a_1, a_2, a_3>0$ and
with $a_k=n$.

For all $n$, there is a maximum $k$ such that $f_k(n)$ is non-zero. Answering a question 
of Spiro~\cite{Spiro}, we show that there is a finite upper bound 
(we specifically prove 561001) on $f_k(n)$
for any positive integer $n\ge 3$ and this maximum 
$k$.

We do this by showing that $f_k(n)$ has transitions in $n$ around constant multiples of $\phi^{3k/2}$
(where $\phi$ is the real root of
$\phi^3=\phi^2+\phi+1$):
there exists a constant $C$ such that $f_k(n)>0$ whenever $n>C\phi^{3k/2}$ and for 
any constant $T$, the values of $f_k(n)$ with $n<T\phi^{3k/2}$ have an upper bound independent of $k$.
\end{abstract}
\section{Introduction}
A {\em tribonacci sequence of length $k$} is a sequence of integers 
$\langle a_i\rangle_{i=1}^k$ such that $a_i=a_{i-1}+a_{i-2}+a_{i-3}$ for all $4\le i\le k$. We say that such a 
sequence {\em terminates} at $a_k$ and that it is {\em positive} if $a_1, a_2, a_3>0$ - note that
this easily implies that $a_i>0$ for all $i$. Denote by 
$f_k(n)$ the number of tribonacci sequences of length $k$ 
terminating at $n$.

Clearly $f_1(n)=1$ for all $n>0$, the only tribonacci 
sequence of length 1 terminating at $n$ being 
$\langle n\rangle$. Further, 
$f_2(n)=f_3(n)=\infty$ as we can choose any values for the
proceeding terms.

For $n\ge 3$, there exists a tribonacci sequence of length longer
than 3 terminating at $n$, for example 
$\langle n-2, 1, 1, n\rangle$. However for any tribonacci 
sequence $\langle a_i\rangle_{i=1}^k$ of length $k$, and for
any $4\le i\le k$, 
$a_i=a_{i-1}+a_{i-2}+a_{i-3}\ge a_{i-1}+2$, so by induction
$a_i\ge 2i-5$ for all $3\le i\le k$, and hence if $n<2k-5$,
$f_k(n)=0$. Let $t(n)$ be the largest number such that 
$f_{t(n)}(n)>0$.

Let $p(n)$ denote the number of positive tribonacci sequences of 
length $t(n)$ terminating at $n$, so $p(n)=f_{t(n)}(n)$.

Clearly, since $t(1)=t(2)=3$ it follows that $p(1)=p(2)=\infty$. 
Spiro~\cite{Spiro} asks
\begin{qn}
Does there exist some absolute constant $c$ such that for all $n\ge 3$, $p(n)\le c$ for all
$n$?
\end{qn}

The purpose of this paper is to give a positive answer to this question. Indeed we will show
\begin{thm}\label{T:barbara}
For any integer $n\ge 3$, there are at most 561001 positive tribonacci sequences of length $t(n)$ 
terminating at $n$.
\end{thm}

It turns out the key question for our proof is the minimum size of the vector
$\begin{pmatrix}a_1\\a_2\\a_3\end{pmatrix}$ where 
$\langle a_i\rangle_{i=1}^n$ is a non-zero tribonacci 
sequence terminating at $a_n=0$. In Section~\ref{S:dalek} we will
show a lower bound on such a sequence of the
order of $\phi^{n/2}$, which will allow us to prove

\begin{thm}\label{T:susan}
For any positive integers $n, k$ with $k\ge 4$, the number
of positive sequences of length $k$ terminating at $n$ is 
at most
\[\lceil1500\frac{n}{\phi^{3k/2}}\rceil^2.\]
\end{thm}

In Section~\ref{S:cyberman} we turn to trying to put an upper bound on numbers that don't
have any positive tribonacci sequences of length $k$ terminating at them. This is an instance of
the Coin Problem, also known as calculating the Frobenius Number. We construct two 
specific tribonacci
sequences terminating at $a_n=0$ with $\begin{pmatrix}a_1\\a_2\\a_3\end{pmatrix}$ being of the 
order of $\phi^{n/2}$ and with the integers $a_1, a_2, a_3$ having specified signs, allowing
us to prove

\begin{thm}\label{T:ian}
For any integer $n$ above $0.2\phi^{3k/2}$, there exists a positive tribonacci sequence of
length $k$ terminating at $n$.
\end{thm}

This will be all that is required.

\begin{proof}[Proof of Theorem~\ref{T:barbara}]
There is no sequence of length $t(n)+1$ terminating at $n$.
Hence by Theorem~\ref{T:ian}, it follows that 
$n<0.2\phi^{3(t(n)+1)/2}=0.2\phi^{3/2}\phi^{3t(n)/2}$.

Thus from Theorem~\ref{T:susan}, it follows that there are at most 
\begin{align*}
\lceil1500\frac{n}{\phi^{3t(n)/2}}\rceil^2
&\le\lceil1500\frac{0.2\phi^{3/2}\phi^{3t(n)/2}}
{\phi^{3t(n)/2}}\rceil^2\\
&\le\lceil300\phi^{3/2}\rceil^2=749^2=561001
\end{align*}
positive tribonacci sequences of length $t(n)$ terminating at $n$.
\end{proof}

In Section~\ref{S:silurian}, we will investigate which recurrence relations of the form
$x_{n}=ax_{n-1}+bx_{n-2}+cx_{n-3}$ for non-negative $a$ and $b$ and for positive $c$ the arguments
in this paper can be carried across to. We will extend the result in the earlier sections to the
following case.

\begin{thm}\label{T:foreman}
Suppose $a,b,c$ are non-negative integers with $a+b>0$, $c=1$ and such that $x^3-ax^2-bx-c=0$ has exactly
one real root. Then there is an absolute bound $T$ such that if positive integers $k\ge 4$ and $n$ are
such that there are no positive sequences $\langle a_i\rangle_{i=1}^{k+1}$ satisfying the
recurrence relation $a_i=aa_{i-1}+ba_{i-2}+ca_{i-3}$ of length $k+1$ terminating at $n$, then there
are at most $T$ such sequences of length $k$ terminating at $n$.
\end{thm}

We will leave open the question of which linear recurrences satisfy this property, but will at least demonstrate an example of a recurrence that does not. In particular we will show the existence of positive
integers $k$ and $n$ such that there is no positive sequence $\langle a_i\rangle_{i=1}^{k+1}$
satisfying the recurrence relation $a_i=a_{i-1}+a_{i-2}+2a_{i-3}$ of length $k+1$ terminating at $n$,
but for which the number of such sequences of length $k$ terminating at $n$ is unbounded.

\section{Lower Bound}\label{S:dalek}
Let us say a sequence $\langle a_i\rangle_{i=1}^\infty$ is a {\em reverse-tribonacci} sequence if for all 
$i\ge 0$, $a_i=a_{i+1}+a_{i+2}+a_{i+3}$. Let us write out the
expression for the reverse-tribonacci sequence starting 
$\langle 0, k, l\rangle$. Recall that $\phi$ is the real solution
to $\phi^3=\phi^2+\phi+1$. We write the complex roots as 
$\phi_1$ and $\phi_2=\overline{\phi_1}$.

\begin{lem}\label{L:benson}
For all integers $k, l$, 
if $\langle a_i\rangle_{i=1}^{\infty}$ is a reverse-tribonacci
sequence with $a_1=0$, $a_2=k$ and $a_3=l$, then for all
$i$, $a_i$ can be expressed as

\begin{align*}
a_i&=\alpha\phi^{-i}+(k\psi_1+l\zeta_1)\phi_1^{-i}+
(k\psi_2+l\zeta_2)\phi_2^{-i}\\
&=(\alpha\phi^{-3i/2}+\beta\cos(\gamma-\delta i))\phi^{i/2},
\end{align*}
where 
\begin{align*}
\psi_1&=\frac{\phi_1^3+\phi_1^2}{\phi_1^2+2\phi_1+3}\\
\psi_2&=\overline{\psi_1}\\
\zeta_1&=\frac{\phi_1^3}{\phi_1^2+2\phi_1+3}\\
\zeta_2&=\overline{\zeta_1}\\
\alpha&=\frac{k\phi^2+(k+l)\phi^3}{\phi^2+2\phi+3}\\
\beta e^{\gamma i}&=2(k\psi_1+l\zeta_1)\textrm{ and}\\
e^{\delta i}&=\phi_1\sqrt{\phi}.
\end{align*}
\end{lem}

\begin{proof}
Any
two-way infinite tribonacci sequence $\langle a_i\rangle_{-\infty}^{\infty}$ can be written as
$a_i=p\phi^i+q\phi_1^i+r\phi_2^i$ for some $p, q$ and $r$.

Thus any reverse-tribonacci sequence $\langle a_i\rangle_{-\infty}^{\infty}$ can be written as 
$a_i=p\phi^{-i}+q\phi_1^{-i}+r\phi_2^{-i}$ for some $p$, $q$, $r$. Solving for the 
$p, q, r$ that give $a_1=0$, $a_2=k$ and $a_3=l$ leads to the above expression.
\end{proof}

Note that in the above expressions, $\psi_1$, $\psi_2$,
$\zeta_1$, $\zeta_2$ and $\delta$ are constants that do not
depend on $k$ and $l$.

\begin{lem}\label{L:victoria}
For any integers $k$ and $l$, if $\alpha$ and $\beta$ are
defined as in Lemma~\ref{L:benson}, then 
$|\alpha|\le|k|+|l|$ and
$\beta\ge\frac{|k|+|l|}{31}$.
\end{lem}

\begin{proof}
$\alpha$ is roughly $0.9546k+0.6184l$, which is clearly bounded above by $|k|+|l|$. 
$\psi_1$ is roughly $0.02267-0.217i$ and $\zeta_1$ is roughly $0.1908-0.0187i$. As such, if 
$k$ and $l$ are non-negative then the real part of $2(k\psi_1+l\zeta_1)$ (and hence $\beta$)
is at least $0.04(k+l)>\frac{|k|+|l|}{31}$.

For $k$ positive and $l$ negative, the minimum value of $\frac{\beta}{k-l}$ is approximately
$0.03221>\frac{1}{31}$, and is achieved around $k=-0.3653l$.
\end{proof}

Finally we need a simple trigonometric property

\begin{lem}\label{L:zoe}
For any real numbers $p$ and $q$ with $\frac{\pi}2<q<\pi$,
the larger of $|\cos(p)|$ and $|\cos(p+q)|$ is at least $\cos(q/2)$.
\end{lem}

\begin{proof}
Note that $\cos(q)<0$. Thus
\begin{align*}
2(\cos(p)^2+\cos(p+q)^2)&=2\cos(p)^2+2\cos(p+q)^2\\
&=\cos(2p)+\cos(2p+2q)+2\\
&=2\cos(2p+q)\cos(q)+2\\
&\ge2+2\cos(q)\\
&=4\cos(q/2)^2.
\end{align*}

Thus either $\cos(p)^2\ge\cos(q/2)^2$ or $\cos(p+q)^2\ge\cos(q/2)^2$.
\end{proof}

This allows us to put a lower bound on the size of at least one of each consecutive pair of a
reverse-tribonacci sequence.

\begin{cor}\label{C:jamie}
Given a non-zero integer reverse-tribonacci sequence 
$\langle a_i\rangle_{i=1}^{\infty}$ with $a_1=0$, for every integer $n\ge 2$, either $|a_n|>0.01\phi^{n/2}$ or 
$|a_{n+1}|>0.01\phi^{(n+1)/2}$ (or both).
\end{cor}

\begin{proof}
For $n\ge 2$ if $a_n$ and $a_{n+1}$ are both 0, then $a_1$ is the same sign as $a_{n-1}$. Since 
$a_1=0$, it follows that the entire series must be 0. Since the sequence is non-zero, it follows
that either $|a_n|\ge 1$ or $|a_{n+1}\ge 1$. Since $1>0.01\phi^{n/2}$ for $n\le 15$, we have proved 
the statement for $n\le 14$. Thus we may assume $n\ge 15$.

By Lemma~\ref{L:benson}, $\frac{a_i}{\phi^{i/2}}$ can be written as 
$\alpha\phi^{-3i/2}+\beta\cos(\gamma-\delta i)$. 

Now by Lemma~\ref{L:zoe}, at least one of $|\cos(\gamma-\delta n)|$ and $|\cos(\gamma-\delta (n-1))|$
is at least $\cos(\delta/2)$ ($\delta=2.176$ is between $\frac{\pi}2$ and $\pi$). Let $t$ be the
choice from $\{n-1,n\}$ that maximises $|\cos(\gamma-\delta t)|$.

By Lemma~\ref{L:victoria}, if
we write $\alpha'=\frac{\alpha}{|k|+|l|}$ and $\beta'=\frac{\beta}{|k|+|l|}$ then $|\alpha'|\le 1$
and $\beta'>\frac{1}{31}$.

Therefore
\begin{align*}
|\frac{a_t}{\phi^{t/2}}|&=|\alpha\phi^{-3t/2}+\beta\cos(\gamma-\delta t)|\\
&\ge|\alpha'\phi^{-3t/2}+\beta'\cos(\gamma-\delta t)|\\
&\ge|\beta'\cos(\gamma-\delta t)|-|\alpha'\phi^{-3t/2}|\\
&\ge\frac{1}{31}\cos(\delta/2)-\phi^{-3t/2}\ge\frac{\cos(\delta/2)}{31}-\phi^{-22.5}>0.01.
\end{align*}
\end{proof}

Then we have a bound on the size of tribonacci sequences terminating at 0.

\begin{cor}\label{C:polly}
For $n\ge 3$, if $\langle a_i\rangle_{i=1}^n$ is a non-zero integer tribonacci sequence 
terminating at 0 then either $|a_1|>0.01\phi^{n/2}$ or
$|a_2|>0.01\phi^{(n-1)/2}$ (or both).
\end{cor}

\begin{proof}
Let $k=a_{n-1}$ and $l=a_{n-2}$. Then if $\langle b_i\rangle_{i=1}^{\infty}$ is the 
reverse-tribonacci sequence with $b_1=0$, $b_2=k$ and $b_3=l$, then $a_i=b_{n+1-i}$ for all 
$1\le i\le n$. Then this is just a restatement of 
Corollary~\ref{C:jamie}.
\end{proof}

This is all we need to prove Theorem~\ref{T:susan}.

\begin{proof}[Proof of Theorem~\ref{T:susan}]
Partition the tribonacci sequences of length 
$k\ge 4$ terminating at $n$ $\langle a_i\rangle_{i=1}^k$ by the 
pair $(\lfloor\frac{a_1}{0.01\phi^{k/2}}\rfloor,
\lfloor\frac{a_2}{0.01\phi^{k-1/2}}\rfloor)$.

If two sequences $\langle a_i\rangle_{i=1}^k$ and 
$\langle b_i\rangle_{i=1}^k$ have the same pair, then
$|a_1-b_1|<0.01\phi^{k/2}$ and $|a_2-b_2|<0.01\phi^{(k-1)/2}$ and
hence, by Corollary~\ref{C:polly}, either 
$\langle a_i-b_i\rangle_{i=1}^k$ is zero everywhere or does
not terminate at 0.

Thus each distinct tribonacci sequence of length $k$ terminating
at $n$ has a distinct pair.

Define tribonacci sequence by $x_1=1, x_2=0, x_3=0$. Then if 
$a_1, a_2, \ldots, a_k$ is a positive tribonacci sequence,
$a_i\ge x_ia_4$ for $i=2, 3$ and $4$ and therefore 
$a_k\ge x_ka_4$. Now $x_k<\phi^k/11$ for all $k\ge 4$ and
hence $a_1+a_2+a_3\le\frac{11n}{\phi^k}$ for all
tribonacci sequences of length $k$ terminating at $n$.

Thus $\lfloor\frac{a_1}{0.01\phi^{k/2}}\rfloor$ is at most
$\frac{1100n}{\phi^{3k/2}}$ and 
$\lfloor\frac{a_2}{0.01\phi^{(k-1)/2}}\rfloor$ is at most
$\frac{1100n}{\phi^{3k-1/2}}<\frac{1500n}{\phi^{3k/2}}$.

It follows that the number of Tribonacci sequences of length
$k\ge 4$ terminating at $n$ is at most
$\lceil\frac{1500n}{\phi^{3k/2}}\rceil^2$.
\end{proof}

Note we have not worked hard here to get the best bound.
In a previous draft we had a much more complicated proof of an
upper bound which showed, in place of Corollary~\ref{C:polly},
that if $\langle a_i\rangle_{i=1}^n$ terminated at 0
then $\sqrt{a_1^2+a_2^2+a_3^2}>0.28\phi^{n/2}$, which led to an
upper bound for the main theorem of 42875.

\section{Upper Bound}\label{S:cyberman}
In this section, we turn to numbers which are not the terminus for any tribonacci sequence of length
$k$, working towards a proof of Theorem~\ref{T:ian}.

To that end, define three infinite tribonacci sequences $\langle p_i\rangle_{i=1}^{\infty},
\langle q_i\rangle_{i=1}^{\infty}$
and $\langle r_i\rangle_{i=1}^{\infty}$ by $(p_1,p_2,p_3)=(1,0,0), (q_1,q_2,q_3)=(0,1,0)$ and 
$(r_1,r_2,r_3)=(0,0,1)$. It is clear that for any tribonacci sequence $\langle a_i\rangle_{i=1}^n$, 
$a_n=a_1p_n+b_1q_n+c_1r_n$. Thus we are simply looking to get an upper bound on the largest number
which cannot be written as a positive integral linear combination of $p_n, q_n$ and $r_n$.
This is called the {\em Frobenius Number} of $p_n, q_n$ and $r_n$.

First let us see that a finite bound does exist.

\begin{lem}
For all $k\ge 1$, $p_k$, $q_k$ and $r_k$ have no non-trivial common divisor.
\end{lem}

\begin{proof}
If $p_k, q_k$ and $r_k$ had a non-trivial common divisor $t>1$ then $t$ would be a common
divisor of the terminus of every tribonacci sequence of length $k$, from which it would
follow that $t$ would in fact be a common divisor of $p_{k+l}$ for all $l\ge 0$ 
(since $\langle p_{i+l}\rangle_1^k$ is a tribonacci sequence of length $k$).

Then, since $p_i=p_{i+3}-(p_{i+1}+p_{i+2})$, it would follow that $t$ would be a common
divisor of $p_{k-1}, p_{k-2}$ and all the way back to $p_0=1$ by induction, causing a contradiction.
\end{proof}

We will use the following bound, which might be originally due to Killingbergtro.

\begin{thm}\label{T:sarah}
Suppose $p, q$ and $r$ are integers with no non-trivial common divisor and let us suppose 
$ap=bq+cr$ and $dq=ep+fr$ where $a,c,d,f>0$ and $b,e\ge 0$. Then for every integer 
$N\ge ap+dq+r$, $N$ can be written in the form $xp+yq+zr$ for some positive integers $p,q,r$.
\end{thm}

\begin{proof}
Let $x, y, z$ be positive integers such that $px+qy+rz$ is equivalent to 
$N\pmod r$, but for which $px+qy+rz$ is minimal (such a triple $x, y, z$ exist because,
as is well known, if $p, q$ and $r$ have no non-trivial common divisor then all sufficiently
large integers can be written in the form $px+qy+rz$, and many of these sufficiently large
integers are equivalent to $N\pmod r$.) 

Since $px+qy+rz$ is minimal, $px+qy+rz-r$ cannot be written as a positive linear combination of
$x, y$ and $z$.

Thus in each of the equations
\begin{align*}
px+qy+rz-r&=px&    &+qy    &&+r(z-1)\\
px+qy+rz-r&=p(x-a)&&+q(y+b)&&+r(z+c-1)\\
px+qy+rz-r&=p(x+e)&&+q(y-d)&&+r(z+f-1),
\end{align*}
it must follow that one of the coefficients must not be positive. Two of the coefficients in
each equation are clearly positive, so it follows that $x\le a, y\le d$ and $z\le 1$, so
$px+qy+rz\le pa+qd+r\le N$. Since $N$ and $px+qy+rz$ are equivalent modulo $r$, 
there exists a non-negative
integer $t$ such that $N=px+qy+rz+rt$. Then $N=px+qy+r(z+t)$.
\end{proof}

Therefore, to show that all sufficiently large integers can be written as the terminus of a
tribonacci sequence of length $k$, we just need to find linear combinations of $p_n, q_n$
and $r_n$ combining to 0, with particular signs of the
combinations. This is equivalent to finding tribonacci sequences ending at 0,
which is equivalent to finding reverse-tribonacci sequences starting at 0,
and hence we can again use the expression from Lemma~\ref{L:benson}, which states that if 
$\langle a_i\rangle_{i=1}^{\infty}$ is a reverse-tribonacci sequence with $a_1=0$, $a_2=k$ and
$a_3=l$ then for all $n$
\[a_n=(\alpha\phi^{-3n/2}+\beta\cos(\gamma-\delta n))\phi^{n/2}.\]

Note that for all but an extremely small collection of $n$, the term 
$\beta\cos(\gamma+\delta n)$ dwarves $\alpha\phi^{-3n/2}$. As such, for a fixed $k$ and $l$,
the sign of
$a_n$ depends only (except for a few very rare counterexamples) on the 
fractional part of $\frac{\delta}{2\pi}n$.

\begin{lem}\label{L:yates}
For each integer $n\ge 4$, there exists a tribonacci sequence 
$\langle a_i\rangle_{i=1}^n$ terminating at $a_n=0$, with $a_1>0$, $0\ge a_2$, $0>a_3$ and with 
$a_1<0.81\phi^{n/2}.$

Similarly for all such $n$, there exists a tribonacci sequence
$\langle b_i\rangle_{i=1}^n$ terminating at $b_n=0$, with $b_2>0$, $0\ge b_1$, $0>b_3$ and with
$b_2<0.64\phi^{n/2}.$
\end{lem}

\begin{proof}
\begin{table}\centering
\caption{Table for Lemma~\ref{L:yates}}\label{tab:yates}
\scriptsize
\begin{tabular}{lrrrrrrrrr}\toprule
$t_0$ &$t_1$ &$k$ &$l$ &$\alpha$ &$\beta$ &$\gamma$ &$x_0$ &$x_1$ &$x_2$ \\\midrule
0    &0.06 &0  &1  &0.6184  &0.3834 &-0.0977  &0.3410 &-0.0500 &-0.1694 \\
0.06 &0.16 &-1 &2  &0.2822  &0.8027 &0.4640 &0.6879 &-0.0515 &-0.2163 \\
0.16 &0.22 &-1 &1  &-0.3362 &0.5200 &0.8677 &0.4526 &-0.0471 &-0.2482 \\
0.22 &0.35 &-1 &0  &-0.9546 &0.4364 &1.6749 &0.3778 &-0.0354 &-0.0294 \\
0.35 &0.45 &-1 &-1 &-1.5731 &0.6360 &2.3067 &0.5517 &-0.0538 &-0.1588 \\
0.45 &0.56 &0  &-1 &-0.6184 &0.3834 &3.0439 &0.3410 &-0.0500 &-0.0548 \\
0.56 &0.66 &1  &-2 &-0.2822 &0.8027 &-2.6776  &0.6879 &-0.0515 &-0.2163 \\
0.66 &0.72 &1  &-1 &0.3362  &0.5200 &-2.2739  &0.4526 &-0.0471 &-0.2482 \\
0.72 &0.85 &1  &0  &0.9546  &0.4364 &-1.4667  &0.3778 &-0.0354 &-0.0294 \\
0.85 &0.95 &1  &1  &1.5731  &0.6360 &-0.8349  &0.5517 &-0.0538 &-0.1588 \\
0.95 &1    &0  &1  &0.6184  &0.3834 &-0.0977  &0.3745 &-0.1864 &-0.0548 \\
\bottomrule
\end{tabular}
\end{table}

We will split into cases based on the fractional part of $\frac{\delta n}{2\pi}=0.3464n$. See
Table~\ref{tab:yates}. For each row, if 
$t_0\le\frac{\delta n}{2\pi}-\lfloor\frac{\delta n}{2\pi}\rfloor\le t_1$,
then for the given values of $k$ and $l$, if $\beta$ and $\gamma$ are as defined in Lemma~\ref{L:benson},
one can confirm that $\beta\cos(\gamma-\delta n)\ge x_0>0.34$, while 
$\beta\cos(\gamma-\delta (n-1))\le x_1<-0.035$ and $\beta\cos(\gamma-\delta (n-2))\le x_2<-0.029$.

Furthermore, for all such $k, l$, $|\alpha|<1.58$, so if $n\ge 7$, 
$|\alpha\phi^{-3(n-2)/2}|\le0.017$, from which it follows that 
$a_n>0>a_{n-1},a_{n-2}$. Further, 
$\frac{a_n}{\phi^{n/2}}<\beta+0.017<0.81.$

For $4\le n<7$, we can verify the sequences $(1,0,-1,0), (2,0,-1,1,0)$ and   
$(2,0,-1,1,0,0)$ satisfy the conditions for $(a_1, a_2, \ldots, a_n)$.

\begin{table}\centering
\caption{Other table for Lemma~\ref{L:yates}}\label{tab:jo}
\scriptsize
\begin{tabular}{lrrrrrrrrr}\toprule
$t_0$ &$t_1$ &$k$ &$l$ &$\alpha$ &$\beta$ &$\gamma$ &$x_0$ &$x_1$ &$x_2$ \\\midrule
0    &0.06 &1  &-1 &0.3362  &0.5200 &2.2739  &-0.3362 &0.4625 &-0.0678 \\
0.06 &0.19 &1  &0  &0.9546  &0.4364 &1.4667  &-0.1176 &0.3862 &-0.0528 \\
0.19 &0.29 &1  &1  &1.5731  &0.6360 &0.8349  &-0.2812 &0.5639 &-0.0791 \\
0.29 &0.41 &0  &1  &0.6184  &0.3834 &0.0977  &-0.1311 &0.3369 &-0.0413 \\
0.41 &0.56 &-1 &1  &-0.3362 &0.5200 &-0.8677 &-0.0714 &0.4625 &-0.0678 \\
0.56 &0.69 &-1 &0  &-0.9546 &0.4364 &-1.6749 &-0.1176 &0.3862 &-0.0528 \\
0.69 &0.79 &-1 &-1 &-1.5731 &0.6360 &-2.3067 &-0.2812 &0.5639 &-0.0791 \\
0.79 &0.91 &0  &-1 &-0.6184 &0.3834 &-3.0439 &-0.1311 &0.3369 &-0.0413 \\
0.91 &1    &1  &-1 &0.3362  &0.5200 &2.2739  &-0.0714 &0.4641 &-0.2528 \\
\bottomrule
\end{tabular}
\end{table}

For the sequence $(b_1, b_2, \ldots, b_n)$, see Table~\ref{tab:jo}. Here, for each row, if 
$t_0\le\frac{\delta n}{2\pi}-\lfloor\frac{\delta n}{2\pi}\rfloor\le t_1$,
then for the given values of $k$ and $l$, if $\beta$ and $\gamma$ are as defined in Lemma~\ref{L:benson},
one can confirm that $\beta\cos(\gamma-\delta n)\le x_0<-0.071$, while 
$\beta\cos(\gamma-\delta (n-1))\ge x_1>0.33$ and $\beta\cos(\gamma-\delta (n-2))\le x_2<-0.041$.

Furthermore, for all such $k, l$, $|\alpha|<1.58$, so if $n\ge 7$, 
$|\alpha\phi^{-3(n-2)/2}|\le0.017$, from which it follows that 
$a_{n-1}>0>a_n,a_{n-2}$.

For $4\le n<7$, we can verify the sequences $(0,1,-1,0), (0,1,-1,0,0)$ and   
$(-1,2,-1,0,1,0)$ satisfy the conditions for $(b_1, b_2, \ldots, b_n)$.
\end{proof}

This then completes our proof.

\begin{proof}[Proof of Theorem~\ref{T:ian}]
Lemma~\ref{L:yates} gives us tribonacci sequences $\langle a_i\rangle_{i=1}^n$ and 
$\langle b_i\rangle_{i=1}^n$ 
terminating at $a_n=b_n=0$. It follows that $a_1p_n+a_2q_n+a_3r_n=0=b_1q_n+b_2q_n+b_3r_n.$

Since $a_1,b_2>0>a_3,b_3$ and $0\ge a_2,b_1$, it follows that we can write 
\begin{align*}a_1p_n&=(-a_2)q_n+(-a_3)r_n\text{ and}\\
b_2q_n&=(-b_1)p_1+(-b_3)r_n\end{align*} satisfying the sign
requirements of Theorem~\ref{T:sarah}, so it follows that every integer $N\ge a_1p_n+b_2q_n+r_n$
can be written in the form $xp_n+yq_n+zr_n$ for some positive integers $x, y$ and $z$, and hence
there exists a positive tribonacci sequence of length $k$ ending at $N$.

By the bounds on $a_1$ and $b_2$ given in Lemma~\ref{L:yates}, we have such a tribonacci sequence for
all $N\ge0.81\phi^{k/2}u_k+0.64\phi^{k/2}v_k+w_k$. Since 
$u_k\le v_k\le w_k<0.11\phi^k$ and 
$0.81\phi^{k/2}+0.64\phi^{k/2}+1<1.74\phi^{k/2}$, it follows 
that such a tribonacci sequence exists for all $N\ge0.2\phi^{3k/2}$ as was required.
\end{proof}

\section{Other cubic recurrences}\label{S:silurian}
For non-negative $a, b, c$ we can ask a similar question for recurrences of the form
$x_n=ax_{n-1}+bx_{n-2}+cx_{n-3}$. Formally, let us define $k_{a,b,c}(n)$ to be the largest $k$
such there is a positive $k$-element solution $\langle x_i\rangle_{i=1}^k$ to the recurrence relation
$x_i=ax_{i-1}+bx_{i-2}+cx_{i-3}$, and define $t_{a,b,c}(n)$ to be the number of positive
$k_{a,b,c}(n)$-element solutions that exist.

If $c=0$, this is a quadratic recurrence, and the problem is already solved. If $a=0, b=0$ and 
$c=1$, the recurrence is $x_n=x_{n-3}$, and $k_{a,b,c}(n)$ is not defined for any $n$.

For all $a,b,c\ge 0$ with $c\ge 1$ and $a+b+c\ge 2$, say that
the recurrence $x_n=ax_{n-1}+bx_{n-2}+cx_{n-3}$ is 
{\em congenial} if there exists a finite bound $B$ such that for all $n$, $t_{a,b,c}(n)=\infty$
or $t_{a,b,c}(n)\le B$.

Firstly let us note that not all recurrences are congenial.

\begin{lem}
The recurrence $x_n=x_{n-1}+x_{n-2}+2x_{n-3}$ is not congenial.
\end{lem}

\begin{proof}
Let $\langle p_n\rangle_{n=1}^{\infty}$, 
$\langle q_n\rangle_{n=1}^{\infty}$ and
$\langle r_n\rangle_{n=1}^{\infty}$ be the solutions to the
recurrence starting with $\langle1, 0, 0\rangle$, 
$\langle0, 1, 0\rangle$ and $\langle 0, 0, 1\rangle$
respectively. Then $x_n=x_1p_n+x_2q_n+x_3r_n$.

Solutions to the recurrence can be split as the sum of 
two parts - a sequence of the form $\langle x^{(1)}_n=2^{n-1}k\rangle$ 
and a sequence of the form $\langle x^{(2)}_n\rangle$ which is
periodic with period 3 with 
$x^{(2)}_1+x^{(2)}_2+x^{(2)}_3=0$. It is then
easy to solve for $k$: 
$x_1+x_2+x_3=x^{(1)}_1+x^{(1)}_2+x^{(1)}_3=7k$, so
$k=\frac{x_1+x_2+x_3}7$.

In particular, if you let $t_n=\frac{2^{n-1}}7$,
$p_n-t_n$ is periodic with period 
$\langle\frac67,-\frac27,-\frac47\rangle$, $q_n-t_n$ with period
$\langle-\frac17,\frac57,-\frac47\rangle$ and $r_n-t_n$ with
period $\langle-\frac17,-\frac27,\frac37\rangle$.

For $n=3t$, $x_n=c(x_1+x_2)+(c+1)x_{n-3}$ and
$x_{n+1}=2(c+1)x_1+(2c+1)(x_2+x_3)$ where 
$c=\frac{2^{3t-1}-1}7$. Then $x_{n+1}$ cannot be equal to 
$(2c+1)(2c+3)$ for positive $x_1, x_2, x_3$ ($x_1$ would have to
be a multiple of $2c+1$ that is positive but less than $2c+1$),
but for all $1\le i\le 4c+4$, if $x_1=i$, $x_2=4c+5-i$ and
$x_3=3$, then $x_n=c(4c+5)+(c+1)3=4c^2+8c+3=(2c+1)(2c+3)$.
\end{proof}

The proofs in this paper can be adapted to show that many
other recurrences are congenial. Let us say a polynomial 
$x^3-ax^2-bx-c$ is {\em affable} if $c=1$ 
and it has exactly one real root, which is more than 1. We
will show that affability leads to congeniality.

For the rest of this section, 
fix an affable polynomial $x^3-ax^2-bx-c$ with 
real root $\eta_1$ and complex roots $\eta_2$ and
$\eta_3=\overline{\eta_2}$. Note that $|\eta_2|=\eta_1^{-1/2}$.

We will make use of the following equivalent to 
Lemma~\ref{L:benson}.

\begin{lem}\label{L:benson2}
Given a sequence $\langle x_i\rangle_{i=1}^{n}$ 
satisfying $x_{i+3}=ax_{i+2}+bx_{i+1}+cx_i$ with 
$x_n=0$, $x_{n-1}=k$ and $x_{n-2}=l$, $x_i$ can be
expressed as 
\[x_i=\sum_{j=1}^3(k\psi_j+l\zeta_j)\eta_j^{n-i}\]
for constants $\psi_j, \zeta_j$ depending only on
$x^3-ax^2-bx-c$, which can be rewritten as

\[
\frac{x_i}{\eta_1^{(n-i)/2}}=\alpha\eta_1^{-3(n-i)/2}
+\beta\cos(\gamma-\delta (n-i))\]
where 
\begin{align*}
\alpha&=k\psi_1+l\zeta_1,\\
\beta e^{\gamma i}&=2(k\psi_2+l\zeta_2)\textrm{ and}\\
e^{\delta i}&=\eta_2\sqrt{\eta_1}.
\end{align*}
\end{lem}

We will follow the steps of the proof of 
Theorem~\ref{T:barbara} for all
recurrence relations corresponding to affable polynomials. 
We will not attempt to give an actual bound. 

We note the following, which will be used in the equivalents
of both Theorems~\ref{T:susan} and~\ref{T:ian}
\begin{lem}\label{L:kamelion}
If real numbers $k$ and $l$ satisfy
$k\psi_2+l\zeta_2=0$, then $k=l=0$.
\end{lem}

\begin{proof}
As $\psi_3=\overline{\psi_2}$ and $\zeta_3=\overline{\zeta_2}$,
$k\psi_3+l\zeta_3=0$ and therefore 
the sequence with $x_n=0$, $x_{n-1}=k$ and $x_{n-2}=l$
can simply be expressed as $x_i=(k\psi_1+l\zeta_1)\eta_1^{n-i}$.

As $0=x_n=k\psi_1+l\zeta_1$, it follows that $x_i=0$ for all
$i$ and therefore $k=l=0$.
\end{proof}

We start by following the proof of Theorem~\ref{T:susan}.

\begin{lem}\label{L:polly2}
There exists an absolute bound $M$ such that for
$n\ge 4$ and
all non-zero integer 
sequences $\langle x_i\rangle_{i=1}^{n}$ 
satisfying $x_{i+3}=ax_{i+2}+bx_{i+1}+cx_i$ and $x_n=0$
either $|x_1|\ge M\eta_1^{n/2}$ or $|x_2|\ge M\eta_1^{n/2}$
(or both).
\end{lem}

\begin{proof}
The set of complex numbers $k\psi_2+l\zeta_2$ for $k, l$
real with $|k|+|l|=1$ is a closed subset of the complex
plane (in fact a hollow parallelogram) which, 
by Lemma~\ref{L:kamelion} does not contain 0. 
As such, there exists a constant $V>0$ such that
for all such $k, l$, $|k\psi_2+l\zeta_2|>V$. Then for all 
real $k, l$ it follows that 
$\beta=|k\psi_2+l\zeta_2|>V(|k|+|l|)$.

Clearly if $U=\max(|\psi_1|,|\zeta_1|)$, $\alpha\le U(|k|+|l|)$.

Pick integer $N$ such that 
$V\cos(\delta/2)-U\eta_1^{-3(N-1)/2}$ is positive.
Note we can do this because $\frac{\pi}2<\delta<\pi$.
Then let $M>0$ be such that 
$V\cos(\delta/2)-U\eta_1^{-3(N-2)/2}>M\eta_1$ and 
$\eta_1^{-N/2}>M$.

Now if $n\le N$, then $M\eta_1^{n/2}<1$ (note that 
$\eta_1>1$ since $1^3<a\times1^2+b\times1+c$) and
$x_1, x_2$ cannot both be 0 (as then $x_n$ would have to
be the same sign as $x_3$ and non-zero).

For $n>N$, we know from Lemma~\ref{L:zoe} that there
exists $t\in\{1,2\}$ such that
$|\cos(\gamma-\delta(n-t))|>\cos(\delta/2)>0$.

For such $t$, $n-t\ge N-1$ and so it follows that
\begin{align*}
\frac{|x_t|}{\eta_1^{(n-t)/2}}
&=|\alpha\eta_1^{-3n/2}+\beta\cos(\gamma-\delta n)|\\
&\ge|\beta\cos(\gamma-\delta n)|-|\alpha|\eta_1^{-3n/2}\\
&\ge V\cos(\delta/2)-U\eta_1^{-3n/2}
&\ge M\eta_1
\end{align*}

and hence $|x_t|\ge M\eta_1^{(n+2-t)/2}\ge M\eta_1^{n/2}$.
\end{proof}

This is enough for the equivalent of Theorem~\ref{T:susan}

\begin{thm}\label{T:susan2}
There exists
a fixed bound $T$ such that for any positive integers $n, k$
with $k\ge 4$, the number of positive sequences
$\langle x_i\rangle_{i=1}^{k}$ satisfying 
$x_{i+3}=ax_{i+2}+bx_{i+1}+cx_i$ and terminating at 
$x_k=n$ is at most 
\[\lceil T\frac{n}{\eta_1^{3k/2}}\rceil^2.\]
\end{thm}

\begin{proof}
There is a fixed $P$ such that for any positive sequence 
$\langle a_i\rangle_{i=1}^{k}$ satisfying the recurrence relation
with $k\ge 4$, $P\eta_1^k(a_1+a_2+a_3)\le a_k$.

Thus for any such sequence terminating at $n$, 
$a_1$ and $a_2$ are bounded above by $\frac{n}{P\eta_1^k}$
and for any two such sequences, by Lemma~\ref{L:polly2},
either the first terms or the second terms differ by at least $M\eta_1^{k/2}$.

Thus the number of such sequences is at most
$\lceil\frac{n}{PM\eta_1^{3k}/2}\rceil^2$.
\end{proof}

Now we proceed to follow the proof of Theorem~\ref{T:ian}.
We will need the following Corollary to Lemma~\ref{L:kamelion}

\begin{cor}\label{C:peri}
Given any interval $0\le x<y\le 2\pi$ within $(0,2\pi)$, 
we can pick non-zero integers $k, l$ for which $x<\gamma<y$.
\end{cor}

\begin{proof}
Lemma~\ref{L:kamelion} says that the set 
$\{k\psi_2+l\zeta_2:k,l\in\mathbb{R}\}$, 
when viewed geometrically as a subset of the complex plane,
is not of dimension 1. Thus it must be the entire complex
plane. Pick $x<z<y$, then there exist real $k, l$ with
$k\psi_2+l\zeta_2=e^{iz}$.

Now let $k_n=\lfloor nk\rfloor$ and $l_n=\lfloor nl\rfloor$. 
The limit as $n$ tends to infinity of 
$\frac{k_n\psi_2+l_n\zeta_2}n$ is $e^{iz}$ and therefore
for all sufficiently large $n$, $\gamma$ (which is
the argument of $\frac{k_n\psi_2+l_n\zeta_2}n$) must be
contained in the open interval $(x,y)$.
\end{proof}
 
We shelve this for the moment and focus on a simple piece 
of trigonometry.

\begin{lem}\label{L:yates2}
For all numbers $\frac{\pi}{2}<\delta<\pi$, there exists 
$t$ such that $\cos(t)>0>\cos(t+\delta),\cos(t+2\delta)$
\end{lem}

\begin{proof}
Pick $t$ such that 
$\frac{\pi}{2}-\delta<t<\frac{3\pi}2-2\delta$. There exists
such a $t$ because $\delta<\pi$.

Since $\delta<\pi$, $-\frac{pi}{2}<\frac{\pi}{2}-\delta<t$.
Similarly since $\frac{\pi}2<\delta$, 
$t<\frac{3\pi}2-2\delta<\frac{pi}{2}$. So $-\frac{pi}{2}<t<\frac{pi}{2}$ and hence $\cos(t)>0$.

Further $\frac{\pi}2<t+\delta<t+2\delta<\frac{3\pi}2$, so
$\cos(t+\delta)$ and $\cos(t+2\delta)$ are negative.
\end{proof}

This leads to the following somewhat technical-seeming lemma.

\begin{lem}\label{L:yates3}
For all numbers $\frac{\pi}{2}<\delta<\pi$, there exists 
an $\epsilon>0$ and finitely
many intervals $\langle(x_i,y_i)\rangle_{i=1}^n$ such that for
all $t$ there exists an interval $(x_i,y_i)$ such that for
all $x\in(x_i,y_i)$, $\cos(t+x)>\epsilon$ and 
$-\epsilon>\cos(t+x+\delta),\cos(t+x+2\delta)$.
\end{lem}

\begin{proof}
Pick a $t$ according to Lemma~\ref{L:yates2}, and let
$\epsilon>0$ be a real number such that $\cos(t)>\epsilon$
and $-\epsilon>\cos(t+\delta),\cos(t+2\delta)$.

Then since $\cos$ is a continuous function, there is an
open region $(l,u)$ around $t$ such that for all
$x\in(l,u)$, $\cos(x)>\epsilon$ and 
$-\epsilon>\cos(x+\delta),\cos(x+2\delta)$.

Let $n$ be an integer such that $\frac{4\pi}{n}<u-l$ and then
define $(x_i,y_i)$ to be $(i\frac{2\pi}n,(i+1)\frac{2\pi}n)$
for $1\le i\le n$.

For all $t$ there is a maximum integer $K$ such that 
$t+K\frac{2\pi}{n}\le l$. Then $l<t+(K+1)\frac{2\pi}n$ by 
maximality, but $t+(K+2)\frac{2\pi}{n}\le l+\frac{4\pi}{n}<u$.

Thus if $(K+1)\frac{2\pi}{n}<x<(K+2)\frac{2\pi}{n}$, $l<t+x<u$
and hence $\cos(t+x)>\epsilon$ and $-\epsilon>\cos(t+x+\delta),
\cos(t+x+2\delta)$.

Since $\cos$ is periodic with period $2\pi$, if $1\le i\le n$
and $i$ is equivalent to $K+1$ modulo $n$, then for all 
$x_i<x<y_i$, $\cos(t+x)>\epsilon$ and $-\epsilon>\cos(t+x+\delta),
\cos(t+x+2\delta)$.
\end{proof}

This leads to the equivalent of Lemma~\ref{L:yates}.
\begin{lem}\label{L:yates4}
There exists a constant $C$ such
that for all $n\ge 4$, there exist sequence 
$\langle a_i\rangle_{i=1}^n$ 
satisfying the recurrence relation and terminating at 0
for which $C\eta_1^{n/2}>a_1>0>a_2,a_3$
\end{lem}

\begin{proof}
Since $\frac{\pi}{2}<\delta<\pi$, we can apply Lemma~\ref{L:yates3} and get 
$\epsilon>0$ and finitely many intervals $(x_i,y_i)$ such that for all
$t$ there exists an interval $(x_i,y_i)$ such that for
all $x\in(x_i,y_i)$, $\cos(t+x)>\epsilon$ and 
$-\epsilon>\cos(t+x+\delta),\cos(t+x+2\delta)$.

By Corollary~\ref{C:peri}, for each such interval
$(x_i,y_i)$, we can choose non-zero integers 
$k_i, l_i$ for which $x_i<\gamma(k_i, l_i)<y_i$.
Let $A$ be some real number such that
$|\alpha(k_i,l_i)|<A$ for all such pairs,
$B>0$ be some real number such that
$|\beta(k_i,l_i)|>B$ and let 
$N$ be such that $A\eta_1^{-3N/2}<B\epsilon$.

Then for any $j\ge N+3$, by the statement of Lemma~\ref{L:yates3}, 
there exists an interval $(x_i,y_i)$ such that for all 
$x\in(x_i,y_i)$, $\cos(x-(j-1)\delta)>\epsilon$ and 
$-\epsilon>\cos(x-(j-2)\delta), \cos(x-(j-3)\delta)$. Since $\gamma(k_i,l_i)\in(x_i,y_i)$, it follows that
\[
\frac{a_1}{\eta_1^{(j-1)/2}}
=\alpha(k_i,l_i)\eta_1^{-3(j-1)/2}+\beta(k_i,l_i)\cos(\gamma(k_i,l_i)-(j-1)\delta)\]
is the sum of a number of absolute value at most 
$A\eta_1^{-3N/2}$ and a number that is at least $B\epsilon$ and
so is positive. Similarly $a_2$ and $a_3$ are negative.
$\frac{|a_1|}{\eta_1^{j/2}}$ is bounded above by 
$2B\epsilon$.

For each value $4\le j\le N+2$, we can just choose any sequence satisfying the bounds. For instance, if $a_j=pa_1+qa_2+ra_3$, we set $a_1=q+r$, $a_2=a_3=-p$. Choose $C$ such that $C>2B\epsilon$ and such that for all $4\le j\le N+2$, the sequences we have
chosen satisfy $a_1<C\eta_1^{j/2}$.
\end{proof}

Similarly we can get the following.

\begin{lem}\label{L:yates5}
There exists a constant $C$ such
that for all $n\ge 4$, there exist sequence 
$\langle b_i\rangle_{i=1}^n$
satisfying the recurrence relation and terminating at 0
for which $C\eta_1^{n/2}>b_2>0>b_1,b_3$ 
\end{lem}

\begin{proof}
Proof entirely analagous to Lemma~\ref{L:yates4}.

For Lemma~\ref{L:yates2}, we need a $u$
such that $\cos(u+\delta)>0>\cos(u),\cos(u-2\delta)$.

Pick $u$ such that $\frac{\pi}{2}-2\delta<u<\frac{-\pi}{2}$. 
There exists such a $u$ because $\delta>\frac{\pi}{2}$.

Since $\delta<\pi$, $\frac{-3\pi}{2}<u<\frac{-\pi}{2}$ and hence
$\cos(u)<0$. Similarly $\frac{\pi}{2}<u+2\delta<\frac{3\pi}{2}$
and hence $\cos(u+2\delta)<0$. 
Finally $\frac{\pi}{2}-\delta<u+\delta<\delta-\frac{\pi}{2}$,
so $-\frac{\pi}{2}<u+\delta<\frac{\pi}2$, so $\cos(u+\delta)>0$.

Then by a method equivalent to Lemma~\ref{L:yates3}
there exists an $\epsilon'>0$ and
finitely many intervals $\langle(x'_i,y'_i)\rangle_{i=1}^m$
such that for all $t$ there exists an interval $(x'_i,y'_i)$
such that for all $x\in(x'_i,y'_i)$, 
$\cos(t+x+\delta)>\epsilon'$ and 
$\epsilon'>\cos(t+x),\cos(t+x+2\delta)$.

We then apply the same method as the proof of 
Lemma~\ref{L:yates4}
\end{proof}

This allows us to prove the equivalent of 
Theorem~\ref{T:ian}.

\begin{thm}\label{T:ian2}
There exists
a real number $U$ such that for any positive integers $n, k$
with $k\ge 4$ and $n\ge U\eta_1^{3k/2}$, there is a positive sequence
$\langle x_i\rangle_{i=1}^{k}$ satisfying 
$x_{i+3}=ax_{i+2}+bx_{i+1}+cx_i$ and terminating at 
$x_k=n$.
\end{thm}

\begin{proof}
Denote by $p_k, q_k$ and $r_k$ the integers such that 
$x_k=p_kx_1+q_kx_2+r_kx_3$ for all such sequences $\langle x_i\rangle_{i=1}^{k}$.

Then since there can be an integer sequence ending at $x_k=1$, there is no non-trivial
common divisor of $p_k, q_k$ and $r_k$. 

Further, by Lemma~\ref{L:yates4} and
Lemma~\ref{L:yates5} there exist integers 
$C\eta_1^{k/2}>a_1>0>a_2,a_3$ and $C\eta_1^{k/2}>b_2>0>b_1,b_3$ for which
$a_1p_k+a_2q_k+a_3r_k=b_1p_k+b_2q_k+b_3r_k=0$. Hence by Theorem~\ref{T:sarah},
for all $n\ge a_1p_k+b_2q_k+r_k$, there is such a sequence terminating at $n$.

Since $(2C+1)\zeta_1^{k/2}>a_1+b_2+1$ and 
$p_k, q_k, r_k>T\zeta_1^n$ for some fixed constant $T$,
it follows that for all $n\ge (2C+1)T\zeta_1^{3k/2}$,
there is such a sequence terminating at $n$.
\end{proof}

Finally we are able to show that all affable polynomials are
congenial.

\begin{proof}[Proof of Theorem~\ref{T:foreman}]
For our polynomial $x^3-ax^2-bx-c$ with $c=1$ and
$a+b>1$ and at most one real root, Theorem~\ref{T:ian2}
has stated the existence of a real number $U$ uch
that for any positive integers $n, k$ with $k\ge 4$
and $n\ge U\eta_1^{3k/2}$ there is a positive sequence
of length $k$ terminating at $n$.

Thus if there is no positive sequence of length $k+1$
terminating at $n$, it follows that $n<U\eta_1^{3(k+1)/2}$.

Then by Theorem~\ref{T:susan2} it follows that the number
of sequences of length $k$ terminating at $n$ is at most
$\lceil T\frac{n}{\eta_1^{3k/2}}\rceil^2<\lceil TU\eta_1^{3/2}\rceil^2.$
\end{proof}

For now we leave open the following question.

\begin{qn}
For which positive integers $a, b, c$ with $c>0$ and $a+b>0$
is the recurrence relation $x_n=ax_{n-1}+bx_{n-2}+cx_{n-3}$
congenial?
\end{qn}

\bibliography{Tribonacci}
\bibliographystyle{ieeetr}
\end{document}